\documentclass[12pt]{amsart}
\usepackage{latexsym,fancyhdr,amssymb,color,amsmath,amsthm,graphicx,listings,comment}

\newtheorem{thm}{Theorem}       \newtheorem{propo}{Proposition}
\newtheorem{lemma}{Lemma}       \newtheorem{coro}{Corollary}
\usepackage[section]{placeins}      \let\paragraph\subsection
\title{Poincar\'e Hopf for vector fields on graphs}
\author{Oliver Knill} \date{11/10/2019}
\address{Department of Mathematics \\ Harvard University \\ Cambridge, MA, 02138 }

\begin{document}

\begin{abstract}
We generalize the Poincar\'e-Hopf theorem $\sum_v i(v) = \chi(G)$ to vector fields on a finite simple 
graph $\Gamma=(V,E)$ with Whitney complex $G$. To do so, we define a directed simplicial complex
as a finite abstract simplicial complex equipped with a bundle map $F: G \to V$ telling which vertex 
$T(x) \in x$ dominates the simplex $x$. The index of a vertex is then $i_F(v)=\chi(F^{-1}v)$.
We get a flow by adding a section map $F: V \to G$. The resulting 
map $G \to G$ is a discrete model for a differential equation
$x'=F(x)$ on a compact manifold. Examples of directed complexes are defined by 
Whitney complexes defined by digraphs with no cyclic triangles or gradient fields 
on finite simple graphs defined by a locally injective function \cite{poincarehopf}. 
Other examples come from internal set theory. The result extends to simplicial complexes 
equipped with an energy function $H:G \to \mathbb{Z}$ that implements a divisor. 
The index sum is then the total energy. 
\end{abstract} 
\maketitle

\section{Poincar\'e-Hopf}

\paragraph{}
Let $G$ be a {\bf finite abstract simplicial complex} with {\bf vertex set} $V=\bigcup_{x \in G} x$.
This means that all sets $x \in G$ are non-empty subsets of a finite set $V$ and if $y \subset x$ 
is non-empty, then $y \in G$. 
The complex $G$ is called a {\bf directed complex} if there is a map $F: G \to V$ such that $F(x) \in x$. 
This generalizes a {\bf digraph} $(V,E)$ as the later defines 
the complex $G=\{ \{v\}, v \in V \}  \cup \{ \{ a,b \}, a \to b \in E \}$ 
with $F(a \to b) = a$. Given a {\bf divisor} in the form of an {\b energy function} $H: G \to \mathbb{Z}$
\cite{EnergizedSimplicialComplexes},
the energy of a subcomplex $A$ of $G$ is defined as $\sum_{x \in A} H(x)$. 
The {\bf total energy} of $G$ is $\chi(G)=\sum_{x \in G} H(x)$. For $H(x)=\omega(x)=(-1)^{{\rm dim}(x)}$ with 
${\rm dim}(x)=|x|-1$, the total energy is the {\bf Euler characteristic} $\chi(G)$ of $G$.

\paragraph{}
Given a directed energized complex $(G,H,F)$, define the {\bf index} 
$$  i(v)=\sum_{x \in F^{-1}(v)} H(x) = \chi(F^{-1}(v))  $$
for $v \in V$. It is the energy of the collection of simplices in $G$ {\bf pointing to $v$}. 
The following result essentially is a simple Kirchhoff conservation law:

\begin{thm}[Poincar\'e-Hopf]
$\chi(G)= \sum_{v \in V} i_F(v)$. 
\label{PH}
\end{thm}

\begin{proof}
Transporting all energies from simplices $x$ to vertices along $F$ does not change
the total energy. 
\end{proof}

\paragraph{}
If $G$ is a one-dimensional complex, a direction $F$ defines
a digraph with {\bf $1$-skeleton complex} $G$ of the Whitney complex of a graph.
Now, $i(v) = 1-\chi(S^-(v))$, where where $S^-(x))$ counts the number of incoming edges. 
Averaging the index over all possible directions produces a curvature 
$K(v) = 1-{\rm deg}(x)/2$. The later does not use the digraph 
structure and the corresponding {\bf Gauss-Bonnet formula} 
$\chi(G) =|V|-|E|=\sum_v K(v)$ is equivalent to the {\bf Euler handshake} formula. 
In general, if $F$ is not deterministic, but a random Markov map, then the index becomes 
curvature, an average of Poincar\'e-Hopf indices.

\paragraph{}
If $\Gamma=(V,E)$ is a digraph without cyclic triangles and $G$ is the
Whitney complex, then $F(x)$ can be defined as the maximal
element in the simplex $x$. The reason is that in that case, the digraph structure 
on a complete graph defines a total order and so a maximal element
in each simplex. The graph $K_3$ with cyclic order $1 \to 2 \to 3 \to 1$ 
demonstrates an example, where no total order is defined by the directions.
It is necessary to direct also the triangle to one of the vertices and
need the structure of a directed simplicial complex. 

\paragraph{}
The Poincar\'e-Hopf formula $\sum_{v} i_F(v) =\chi(G)$ in particular applies 
when a locally injective function $g$ on $V$ is given. We can then chose
$F(x)$ to be the vertex in $x$ on which $g$ is maximal \cite{poincarehopf}.
In that case 
$$  i_F(x) = 1-\chi(S^-(x)) \; , $$ 
where $S^-(x)$ is the subgraph of the unit sphere $S(x)$ generated
by the vertices $w$ in the unit sphere $S(x)$ for which $g(w)$ is 
smaller than $g(v)$. If one averages over all possible colorings $g$,
then the index expectation becomes the curvature
\cite{cherngaussbonnet, indexexpectation,indexformula,eveneuler}.

\paragraph{}
The result in Theorem~(\ref{PH}) is more general than \cite{poincarehopf}. 
We could for example assign to each simplex an injective function $g_x : x \to \mathbb{N}$ which 
orders the vertices in $x$. This is more general as the order on a 
sub-simplex $y$ of $x$ does then not have to be compatible with the 
order of $x$.  We will see below that the result in particular applies to digraphs which 
have no cyclic triangles. 

\section{Dynamics}

\paragraph{}
If we complement the map $F: G \to V$ with a map $F: V \to G$
with the property that $v \in F(v)$, then we call this a {\bf discrete vector field}.
If the idea of $F:G \to V$ was thought of as a {\bf vector bundle map} from the fibre bundle $G$ 
to the base $V$, then $F: V \to G$ corresponds to 
a {\bf section of that bundle}. Now $T=F^2:V \to V$ is a map on the finite set $V$ and 
$T=F^2: G \to G$ is a map on the finite set $G$ of simplices, which define a 
flow which is continuous in the sense that the sets $x \in G$ and $T(x) \in G$ 
are required to intersect. The orbit is a path in the {\bf connection graph} of $G$
\cite{Unimodularity,EnergyComplex}
The fact that we have now a dynamical system is the discrete analogue of the 
Cauchy-Picard existence theorem for differential equations on manifolds.

\begin{coro}[Discrete Cauchy-Picard]
A vector field $F$ defines a map $T: G \to G$, in which $x,T(x)$ intersect or
a map $T: V \to V$ in which $v,T(v)$ are in a common simplex. 
\end{coro}
\begin{proof}
The composition of two maps $F:G \to V, V \to G$ is a map $G \to G$. It also
defines a map $V \to V$. 
\end{proof} 

\paragraph{}
If $F$ comes from a digraph, then a natural choice for $F$ is to assign to a vertex $v$ 
an outgoing edge attached to $v$. The map $F: G \to V$ is determined by the direction
on the edge. The vector field can be seen as as section of
a ``unit sphere bundle". This model is crude as the map $F:G \to V$ leads to 
{\bf fibres} $F^{-1}(v)$ which are far from spheres in general
and even can be empty. However, the picture of having a fibre bundle is natural
if one sees $F: V \to G$ as a {\bf section of a fibre bundle}.
In order to have a rich ``tangent space" structure as a fibre, 
it is useful in applications to have a large dimensional simplicial complex so that
the fibres $F^{-1}(v)$ are rich allowing to model a differential equation well. This
picture is what one takes internal set theory seriously. 

\paragraph{}
In {\bf internal set theory}, compact sets can be exhausted by finite sets and continuous
maps are maps on finite sets and the notion of ``infinitesimal" is built in axiomatically into the system. 
This makes sense also in physics, as 
nature has given us a lot of evidence that very small distances are no more
resolvable experimentally. We stick to the continuum because it is a good idealization.
But when thinking about mathematics, already mathematicians like Gregory, Taylor or Euler thought 
in finite terms but chose the continuum as a good language to communicate and calculate. 

\paragraph{}
The theory of ordinary and partial differential equations shows that
dealing with finite models (in the sense of numerical approximations for example) can be 
messy and subtle. It is not the dealing with finite models which is difficult 
but the {\bf translation from the continuum to the discrete} which poses challenges. 
Numerical approximations should preserve various features like integrability or
symmetries which can be a challenge. Rotational symmetries for example
are broken in naive finite models and integrals are not preserved when doing numerical
computations. In internal set theory, all properties are present as the language is not
a {\bf restriction} to a finite model but a {\bf language extension} which allows to 
treat the continuum with finite sets. 

\paragraph{}
In order to get interesting dynamical systems, we need to impose more conditions. 
We want to avoid for example $F(v)=x$ and $F(x)=v$ as this produces a loop
which ends every orbit reaching $v$. We also would like to implement some kind of 
direction from $x \to F(x)=v \to y=F(v)=F^2(y)$. How to model this
appears in {\bf topological data analysis}, where a continuous system is modeled
by a collection of simplicial complexes. Given a compact metric space $(M,d)$, and 
two numbers $\delta>0$ and  $\epsilon>0$, one can look at a finite collection $V$ 
of points which are $\epsilon$ dense in $M$, then define
$G=\{ x \subset V \; | \;$, all points in $x$ have distance $\leq \delta \}$. This is 
a finite abstract simplicial complex. In the manifold case, the cohomology of $G$
is the same than the cohomology of $M$ if $\epsilon$ is small enough with respect to
$\delta$ and the later small enough to capture all features of $M$. 
Persistent properties remain eventually are independent of $\epsilon$ and $\delta$.
In that frame work, the simplicial complexes has of much higher dimension 
than the manifold itself. But the complex is homotopic to a nice triangulation of $M$. 
The framework in topological data analysis is called persistent homology 
(see i.g. \cite{Dlotko2014}).

\paragraph{}
A special case of a directed simplicial complex is a {\bf bi-directed simplicial complex}. 
This is a finite abstract simplicial complex in which two maps $F^+,F^-: G \to V$ are given
with $F^+(x) \in x$ and $F^-(x) \in x$. Think of $F^-(x)$ as the beginning of the simplex and
$F^+(x)$ as the end. Equilibrium points are situations where $F^+(x)=F^-(x)$. These are points
where various orbits can merge.  Every zero dimensional
simplex is an equilibrium point now because $F^+(\{v\}) = F^-(\{v\})=v$ is forced. 
For a directed graph, a natural choice is to define 
$F^+( (a,b) ) = b$ and $F^-( (a,b) ) =a$.

\paragraph{}
If an order is given by a locally injective function $g$ on the vertex set. 
Now, $F^+(x)$ is the largest element in $x$ and $F^-(x)$ is the smallest. The flow produced
by this dynamics is the {\bf gradient flow}. The reverse operation is the gradient flow for 
$-g$. A bidirected simplicial complex defines a vector field as defined above as 
a section of the bundle. The map $T$ produces a {\bf permutation} on the 
{\bf $\omega$ limit sets} $\Omega^+(x) = \bigcap_k \{ T^k(x) \; | \; k >0 \}$ and 
$\Omega^-(x)= \bigcap T^{-k}(G) = \bigcap_k \{ T^k(x) \; | \; k<0 \}$. 
The gradient flow illustrates that these sets are in general different.

\paragraph{}
Given a bi-directed simplicial complex $G$,
the map $F^+$ defines an index $i^+$ and the map $F^-$ defines an index $i^-$. 
One can now look at the {\bf symmetric index} $[i^+(x) + i^-(x)]/2$ which again
satisfies the Poincar\'e-Hopf theorem. In nice topological situations where $M$ 
is a smooth manifold and $F$ a smooth vector field with finitely many 
hyperbolic equilibrium points, the symmetric index agrees with both $i^+$ and
$i^-$ as the sets $S^-(x)$ of incoming simplices or $S^+(x)$ of outgoing simplices
are spheres. In two dimensions for example, we have at
a source or sink that one of the sets $S^{\pm}(x)$ is a circle and the other empty (a $(-1)$-
dimensional sphere), leading to index $i^{\pm}(x)=1-\chi(S^{\pm}(x)) = 1$. At a hyperbolic point with one
dimensional stable and one dimensional unstable manifold, we have $S^{\pm}(x)$ both being zero dimensional
spheres ($2$ isolated points) of $\chi(S^{\pm}(x))=2$ so that $i^{\pm}(x)=-1$. 

\section{Relation with the continuum}

\paragraph{}
Can one relate a differential equation $x'=F(x)$ on a manifold $M$ with a discrete vector
field? The answer is ``yes". But unlike doing the obvious and triangulate
a manifold and look at a discrete version, a better model of $M$ uses a much higher dimensional 
simplicial complex. This is possible in very general terms by tapping into the fundamental 
axiom system of mathematics like ZFC. In 1977, Ed Nelson extended the basic axiom system of mathematics
with three more axioms `internalization' (I), ` standardisation' (S) and `transfer' (T) to 
get a consistent extension ZFC+IST of ZFC in which one has more language and where compact 
topological spaces like a compact manifold are modeled by finite sets $V$. 
See for example \cite{Nelson77,Robert,Nelson,LawlerInternalSetTheory,NelsonSimplicity}). 
The flow produces a permutation on $V \subset M$ but a set $V$ alone without topological reference 
is a rather poor model, even if $M$ is given as a triangulation. Note that also the manifold 
structure is not lost in such a frame work as the set $\{ y \in M \; | \; r/2 \leq |x-y| \leq 2r \}$ 
is homotopic to a $d-1$ sphere if $r$ and $\epsilon$ are infinitesimal. This model of a sphere allows
an action of the rotation group. The genius of Nelson's approach is that it does not require
our known mathematics but resembles what a numerical computation in a finite frame work like a 
computer does. 

\paragraph{}
In order to model an ODE $x'=F(x)$ on a compact manifold $M$ using finite sets, pick 
a finite set $V$ in $M$ such that for every $w \in M$, there is $v \in V$ that is infinitesimally 
close to $w$ (this means that the distance $|v-w|$ is smaller than any ``standard" real number, 
where the attribute ``standard" is defined precisely by the axioms of IST). 
Let $X$ be the simplicial complex given by the set of sets 
with the property that for every $x \in X$, there is $v \in V$ such that every $w \in x$ is infinitesimally 
close to $v$. The set $\{ \overline{x} \; | \; x \in G \}$ is a set of compact subsets
of $M$ whose closure is compact the Hausdorff topology of subsets of $M$ and again by IST can be
represented by a finite set $Y$ of finite sets. Now, $G= \{ x \cap V | x \in Y \}$ is a finite
abstract simplicial complex. 

\paragraph{}
As a computer scientist we can interpret $x \in G$ as a set of points in $M$ 
which are indistinguishable in a given floating point arithmetic (or a given accuracy threshold) 
and $v = F(x)$ as a particular choice of the equivalence class of points which the computer architecture 
gives back to the user. The vector field computation using Euler steps $v \to T(v)=v+dt F(v)$ produces 
then an other equivalence class
of real vectors close to $v$ and the map $T: V \to V$ models the flow of the ODE $x'=F(x)$ on $M$ in the sense
that for any standard $T$, we are after $T/dt$ steps infinitesimally close to the real orbit $x(t)$ given
by the Cauchy-Picard existence theorem for ordinary differential equations.  See \cite{Lanford1986}
for more about the connection. 

\paragraph{}
This simplex model is a more realistic model than a permutation of a finite set of point in $M$ as the
simplicial complex provides an way to talk about objects in a tangent bundle.  
It also has the feature that the Euler characteristic of $M$ is equal to the Euler characteristic of $G$, 
even so the dimension of $G$ modeling the situation can be huge. The classical Poincar\'e-Hopf index of an 
equilibrium point of the flow in an open ball is the same than the sum of the Poincar\'e-Hopf indices of the
vector field on the directed complex. 

\begin{figure}[!htpb]
\scalebox{0.4}{\includegraphics{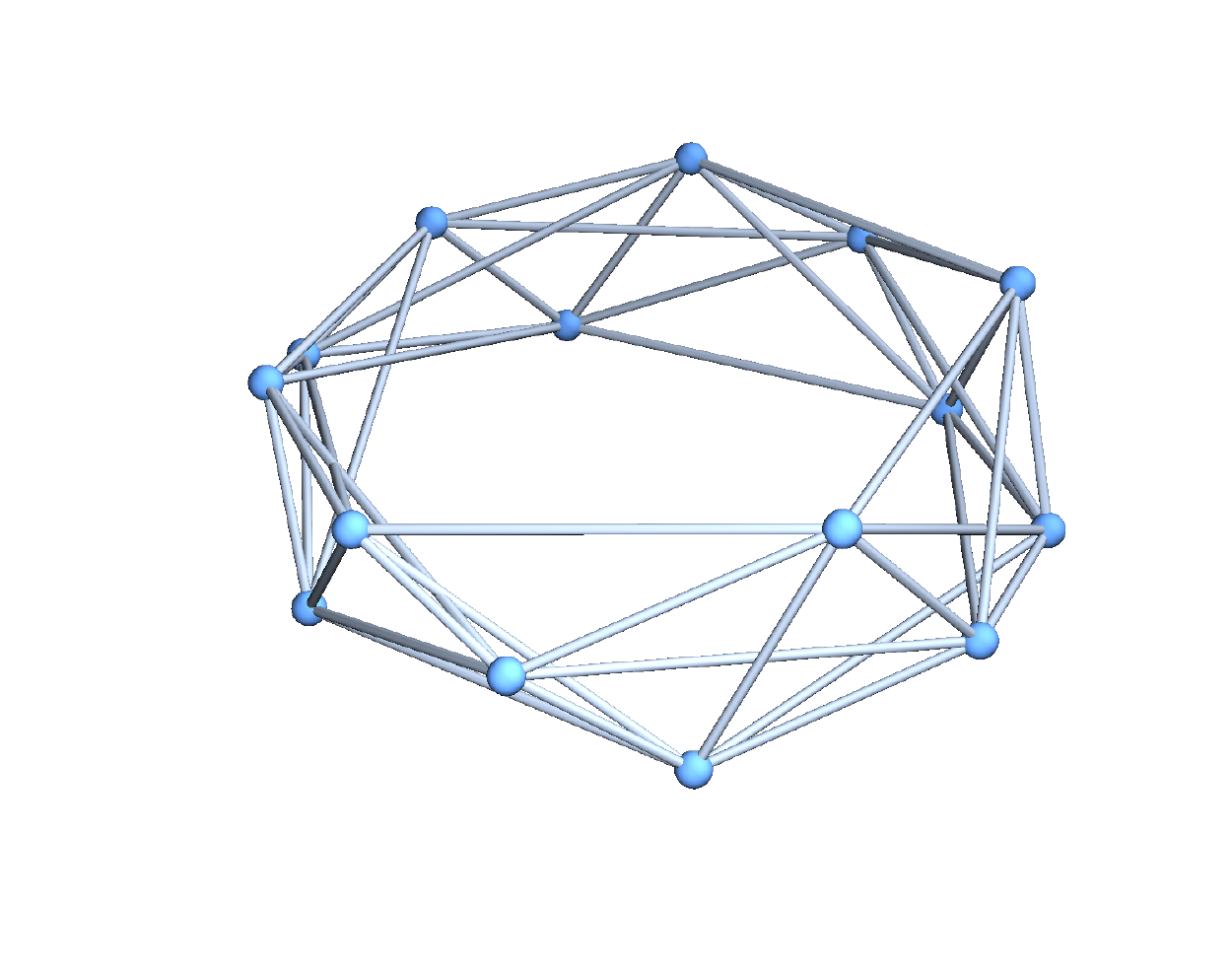}}
\scalebox{0.4}{\includegraphics{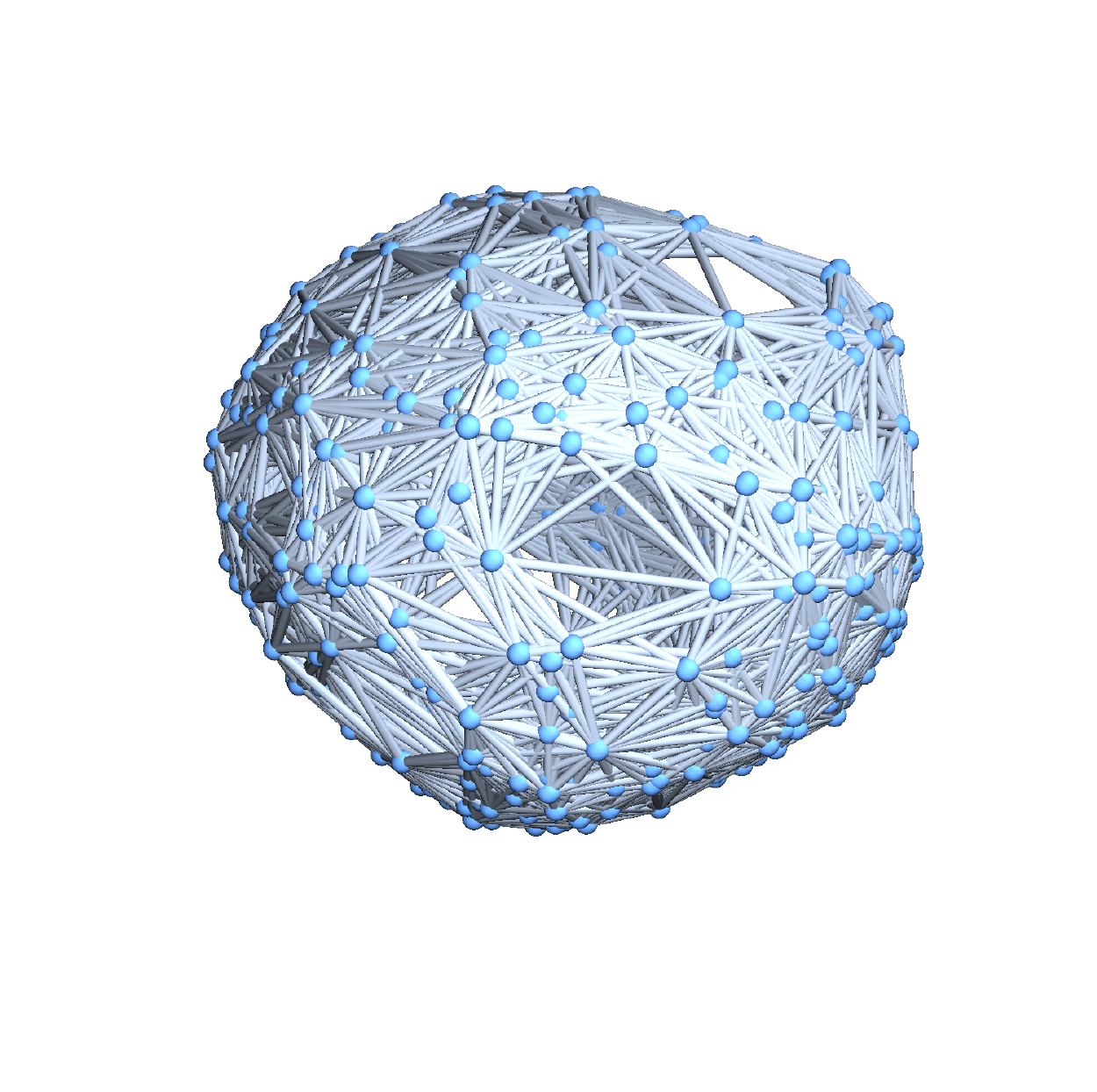}}
\label{triangle}
\caption{
A non-standard model of a $d$-manifold is obtained by picking a large set $V$ of points 
and an $\delta>0$ which models machine precision and where $|V|$ is much larger than $1/\delta^d$.
Define the graph $\Gamma=(V,E)$, where $(a,b)$ is in $E$ if $|b-a| < \delta$. The Whitney complex $G$
of this graph has a large dimension but a geometric realization is homotopic to $M$ and
the Euler characteristic of $G$ is the same than the Euler characteristic of $M$.  We see in this
figure a model for a circle and for a sphere (where 500 random points were chosen and 
connected if Euclidean distance is less than $1/2$).
}
\end{figure}

\section{Remarks} 

\paragraph{}
The following general remark from \cite{NelsonSimplicity} is worth citing verbatim. It is
remarkable as goes far beyond just modeling differential equations and somehow justifies an
simple approach to the question of modeling ordinary differential equations 
using combinatorial models. \\

\paragraph{}
{\it Much of mathematics is intrinsically complex, and there is delight to be
found in mastering complexity. But there can also be an extrinsic complexity
arising from unnecessarily complicated ways of expressing intuitive mathematical
ideas. Heretofore nonstandard analysis has been used primarily to simplify
proofs of theorems. But it can also be used to simplify theories. There are
several reasons for doing this. First and foremost is the aesthetic impulse, to
create beauty. Second and very important is our obligation to the larger scientific
community, to make our theories more accessible to those who need to use them.
}

\paragraph{}
The motivations for the above set-up comes from various places: first of all
from discrete Morse theory \cite{Forman1999,forman98,FormanVectorfields} which 
stresses the importance of using simplicial complexes for defining vector fields (but defines
combinatorial vector fields quite differently), simplicial sets 
\cite{RomeroSergeraert} (which also is a very different approach) as well as digital
spaces \cite{HermanDigitalSpaces,I94a,Evako2013} which provides valuable frame work. 
There is also motivation from the study of discrete approximations to chaotic 
maps \cite{Ran74,ZhVi88,lanford98}, previous theorems of Poincar\'e-Hopf type 
\cite{poincarehopf,knillcalculus,KnillBaltimore,AmazingWorld}, as well as internal set theory 
\cite{Nelson77,Robert,Nelson}. We might comment on the later
more elsewhere. Let us illustrate this with a tale of chaos theory. 
But not with a differential equation but with a map.

\paragraph{}
The {\bf Ulam map} $T(x)=4x(1-x)$ is a member in the {\bf logistic family} on the interval $[0,1]$. 
It is topologically conjugated to the tent map and measure theoretically conjugated to the map $x \to 2x$ on 
the circle $\mathbb{T}^1$ and so {\bf Bernoulli}, allowing to produce IID random variables. 
On any present computer implementation, the orbits of $T(x)=4x(1-x)$ 
and $S(x)=4x-4x^2$ disagree: $S^{60}(0.4) \neq T^{60}(0.4)$. 
Mathematica for example which computes with 17 digits machine precision,
\begin{center}${\rm T[x\_]:=4x(1-x);S[x\_]:=4x-4x^2; \{NestList[T,0.4,60],NestList[S,0.4,60]\}}$ \end{center}
we compute $S^{60}(0.4)=0.715963$ and $T^{60}(0.4)=0.99515$. 
%  T[x_]:=4x(1-x); S[x_]:=4x-4x^2; Table[Last[NestList[T,0.4,k]]==Last[NestList[S,0.4,k]],{k,60}]
More  dramatically, while $T^4(0.4)=S^4(0.4)$, the computer algebra system reports that $T^5(0.4) \neq S^5(0.4)$. Already
computing the polynomial $T^5$ of degree $32$ produces different numbers because the integer coefficients of $T^5$ have
more than 17 digits, while $T^4$ still has less than 17 digits so that the computer can still with it faithfully. 
With a Lyapunov exponent of $\log(2)$, the error of $T$  or $S$ gets doubled in each step so that
$2^{60} \times \epsilon$ is larger than $10$ for the given machine precision $\epsilon=10^{-17}$ and because
$4x-4x^2$ is only in the same $10^{-17}$ neighborhood of $4x(1-x)$, we get different numbers after 60 iterations. 
Floating point arithmetic does not honor the distributivity law. 

\paragraph{}
The simplicial complex $G$ in this case is any set of ``machine numbers" in $[0,1]$ with the property that their mutual distance
is smaller than $10^{-17}$. We have given the example of a map, but the same can happen for differential equations
like the Lorenz system in $\mathbb{R}^3$, where techniques like finding horse shoes enable to prove that the system has invariant
measures on the Lorenz attractor which the dynamics is a Bernoulli system. Also here, 
integrating the differential equation is better modeled microscopically as a map $F$ from a finite set $V$ to a 
simplicial complex $x \in G$ on $V$, rather then selecting a choice $F(x) \in V$. Yes, this defines a deterministic map on $V$
but it also defines a deterministic map on $G$, where the transition $x \to y$ is possible if $x$ and $y$ intersect.
Choosing projections $G \to V$ by picking a point in the center is a good model what happens in a real computer.

\paragraph{}
When investigating definitions of vector fields on graphs, we also aim for a Lie algebra structure.
We have defined in \cite{CartanMagic} an algebraic notion of vector field $X$ as a rule $i_X: \Lambda^{p+1} \to \Lambda^p$ 
balancing the exterior derivative $d: \Lambda^p \to \Lambda^{p+1}$ on discrete forms producing a Lie derivative
$L_X = i_X d + d i_X= D_X^2$ for which the wave equation $u_{tt} = L_X u$ is solved by a d'Alembert
type Schr\"odinger solution $e^{i D_X t}$ which classically is a Taylor theorem for transport. The point of the 
magic formula of Cartan is that it produces a transport flow on $p$ forms. If $i_X$ is replaced by $d^*$
we get the wave equation on $p$-forms. In spirit this is is similar as the flow of $L_X$ produces a dynamics using
neighboring spaces of simplicial complexes. This is also the case for the Laplacian 
$L = d^* d + d d^*$, where now the deterministic $i_X$ is replaced by a diffusion process defined by $d^*$.
So, also in this model, there is a deterministic flow version $L_X$ which classically is given by 
a vector field $X$ or a diffusive flow version $L$, in which case in some sense, we average over all possible
vector fields. The dynamics approach via discrete Lie derivatives is different from the combinatorial one considered
here and there is no Poincar\'e-Hopf theorem yet in the Lie derivative case. 

\paragraph{}
Poincar\'e-Hopf was generalized \cite{parametrizedpoincarehopf} to a function identity 
$$   f_G(t) = 1+t \sum_{x \in V} f_{S_g(x)}(t) $$ 
for the {\bf generating function}
$f_G(t)=1+f_0 t + \dots + f_{d} t^{d+1}$ encoding the
{\bf $f$-vector} of $G$, where $f_k$ is the number of $k$ dimensional sets in $G$. 
Averaging the index over some probability space of functions
gives then curvature $K(x)$ satisfying Gauss-Bonnet $\chi(G) = \sum_x K(x)$. 
\cite{indexexpectation,indexformula,eveneuler}.
In the case of a uniform measure, the Gauss-Bonnet formula generalizes to the functional
formula $f_G(t) = 1+\sum_{x} F_{S(x)}(t)$, where $F_G(t)$ is the anti-derivative of
$f_G$ \cite{cherngaussbonnet,dehnsommervillegaussbonnet}.

\begin{figure}[!htpb]
\scalebox{0.2}{\includegraphics{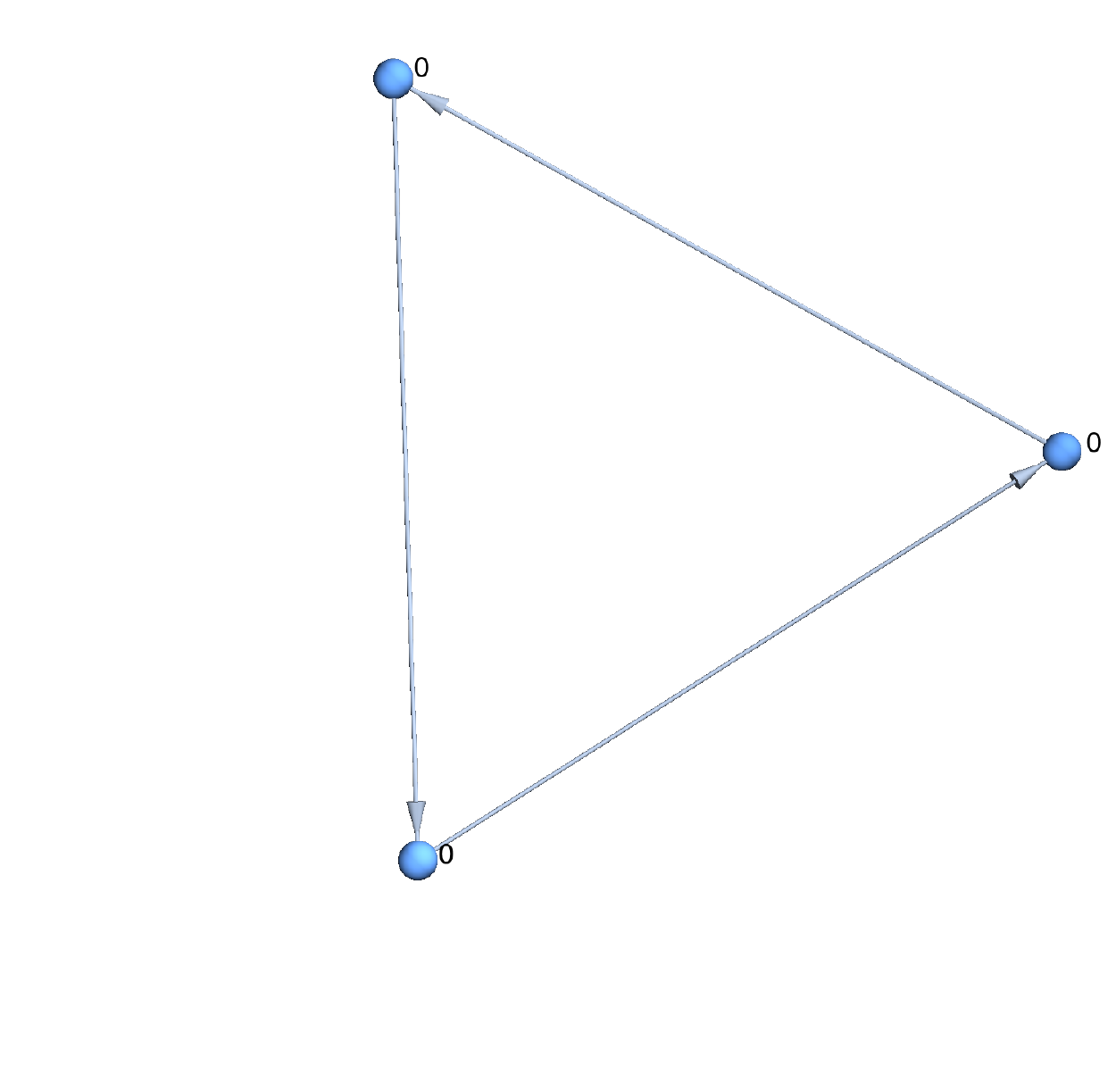}}
\scalebox{0.2}{\includegraphics{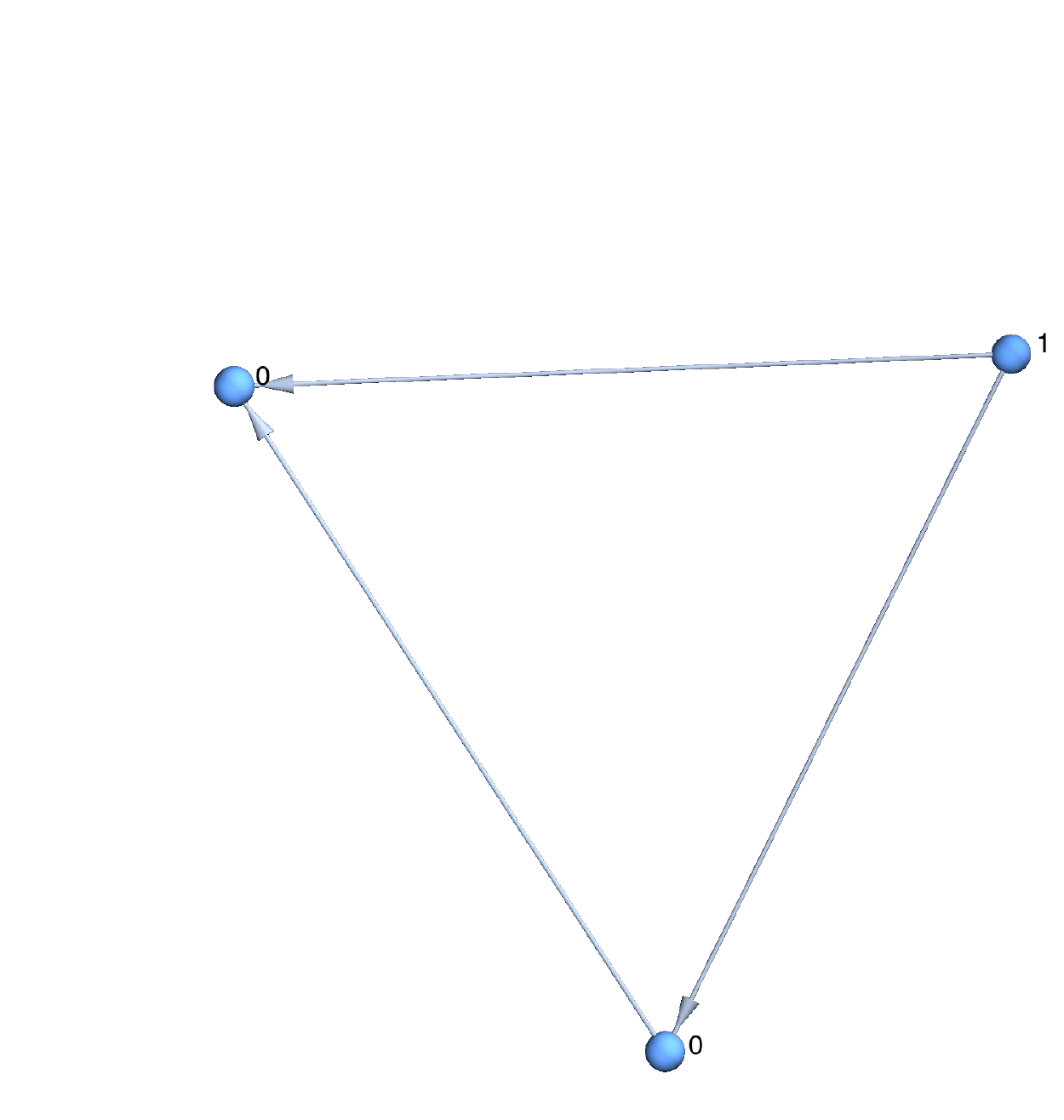}}
\label{triangle}
\caption{
A cyclic triangle has zero integer Poincar\'e-Hopf index at every vertex as it has
to add up to $1$ and symmetry requires the index to be the same on each vertex.
If the vector field is acyclic, it comes from a gradient field, then an index function exists.
It is $1$ at the minimum and $0$ else.
}
\end{figure}

\section{Irrotational digraphs}

\paragraph{}
This note started with the quest to generalize
Poincar\'e-Hopf theorem \cite{Poincare1885,HopfVectorfields,Spivak1999} from 
gradient fields to more general vector fields. This problem has already been posed
in \cite{poincarehopf} and it was pointed out that using a digraph alone can not work
because a circular triangle produces an index inside. It turns out that the existence of
such triangles is the only obstacle. We call digraphs without circular triangles 
{\bf irrotational}. Such graphs still can have circular loops. We just don't want to have
them microscopically to happen on triangles. 

\paragraph{}
Let $\Gamma=(V,E)$ be a finite simple graph and $G$ the Whitney complex consisting of all
finite simple subgraphs. If the edges of $\Gamma$ are oriented, we have a directed graph which 
is also called a digraph. 
Let $F$ be a $1$-form, a function from edges to $\mathbb{R}$ and the curl $dF=0$ is zero 
on all triangles, we can define 
$$   i_F(x) = 1-\chi(S_F(x)) \; , $$
where $S_F(x)=\{ y \in S(x)$, for which $F$ points from $x$ to $y \}$. As we only need the 
direction of $F$ and not the magnitudeof the field, all what matters is a 
{\bf digraph structure} on $G$. 

\paragraph{}
Let us call a finite simple digraph {\bf irrotational} if there are no closed loops of 
length $3$ in the digraph. 
This does not mean that the graph is acyclic. Indeed, any one-dimensional graph (a graph with 
no triangles) and especially the cyclic graphs $C_n$ with $n \geq 4$ are irrational  because there
are no triangles to begin with. For an irrational graph we can then define a {\bf positive function} 
$F: E \to \mathbb{R}$ on the oriented edges of the graph 
such that the curl satisfies $dF=0$. This justifies the name irrotational for the directed graph. 

\paragraph{}
Even if $F$ has zero curl, the field $F$ does not necessarily come from a scalar field $g$, as the 
first cohomology group $H^1(G)$ is not assumed to be trivial. Indeed $H^1(F)$ is the quotient of the
vector space $\{ F \;  dF=0 \}$ divided by the vector space $\{ F=dg \}$ of gradient fields. More generally, the 
maps $G \to H^k(G)$ are functors from finite simple graphs to finite dimensional vector spaces and 
$b_k(G) = {\rm dim}(H^k(G))$ are the {\bf Betti numbers} of $G$ which satisfy $\sum_k (-1)^k b_k = \chi(G)$. 
The Euler characteristic of a digraph is the Euler characteristic of the finite simple graph in 
which the direction structure on the edges is forgotten. 

\paragraph{}
But Poincar\'e-Hopf still works in the irrotational case. 
The Euler characteristic of when the digraph structure is discarded. 

\begin{propo}[Poincar\'e-Hopf for irrotational digraph]
Let $G$ be a finite simple irrotational digraph, then 
the index $i(x) = 1-\chi(S^-(x))$ satisfies $\sum_x i(x) = \chi(G)$. 
\end{propo} 

\begin{proof} 
Let $v$ be a vertex in $G$. The unit ball $B(v)$ is contractible and so has
trivial cohomology. Since the digraph structure is irrotational, we can place function values
$F(e)$ on the oriented edges of $G$ such that the curl of $F$ is zero in 
every triangle of $B(v)$. The field $F$ is therefore a gradient field. Since $H^1(B(v))=0$,
there is a function $g$ such that the gradient of $g$ is $dg=F$. Now, the index at $v$ is the
same than the usual index of $g$.
\end{proof} 

\paragraph{}
Example: The cyclic digraph graph $G=( \{ 1,2,3,4 \}, \{ (12),(23),(34),(41) \})$ 
has $i(x)=0$ for all $x$. In contrast, for
a gradient field $F=dF$ we always had at least two critical points: the local minima has 
index $1$ while the local maxima has index $-1$. In general, there are vector fields 
on any $d$ dimensional flat torus which have no critical point at all.

\paragraph{}
The example of a triangle which has Euler characteristic $1$ shows that one 
can not get a local integer-valued index function on the vertices $x$ alone if
one insists the index function to be deterministic and depend only on the 
structure of $F$ of the unit ball of $x$. 

\begin{figure}[!htpb]
\scalebox{0.5}{\includegraphics{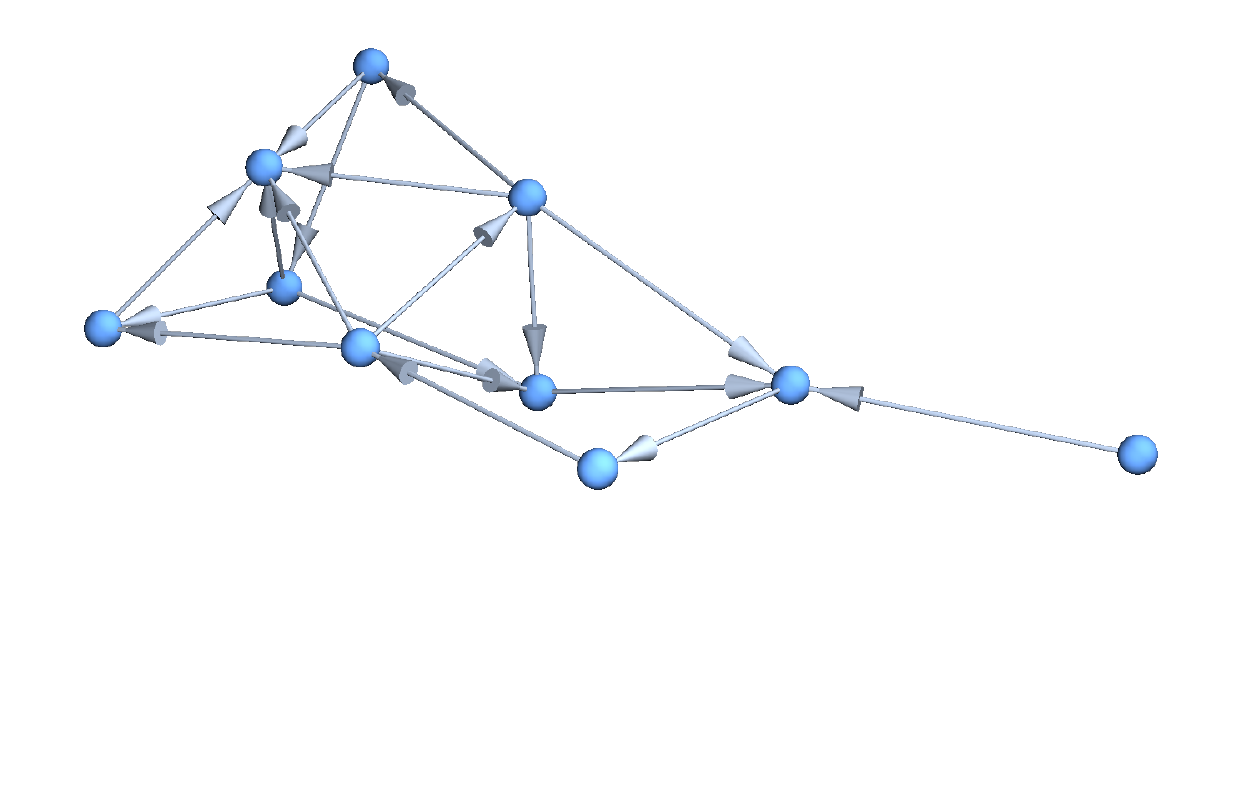}}
\scalebox{0.5}{\includegraphics{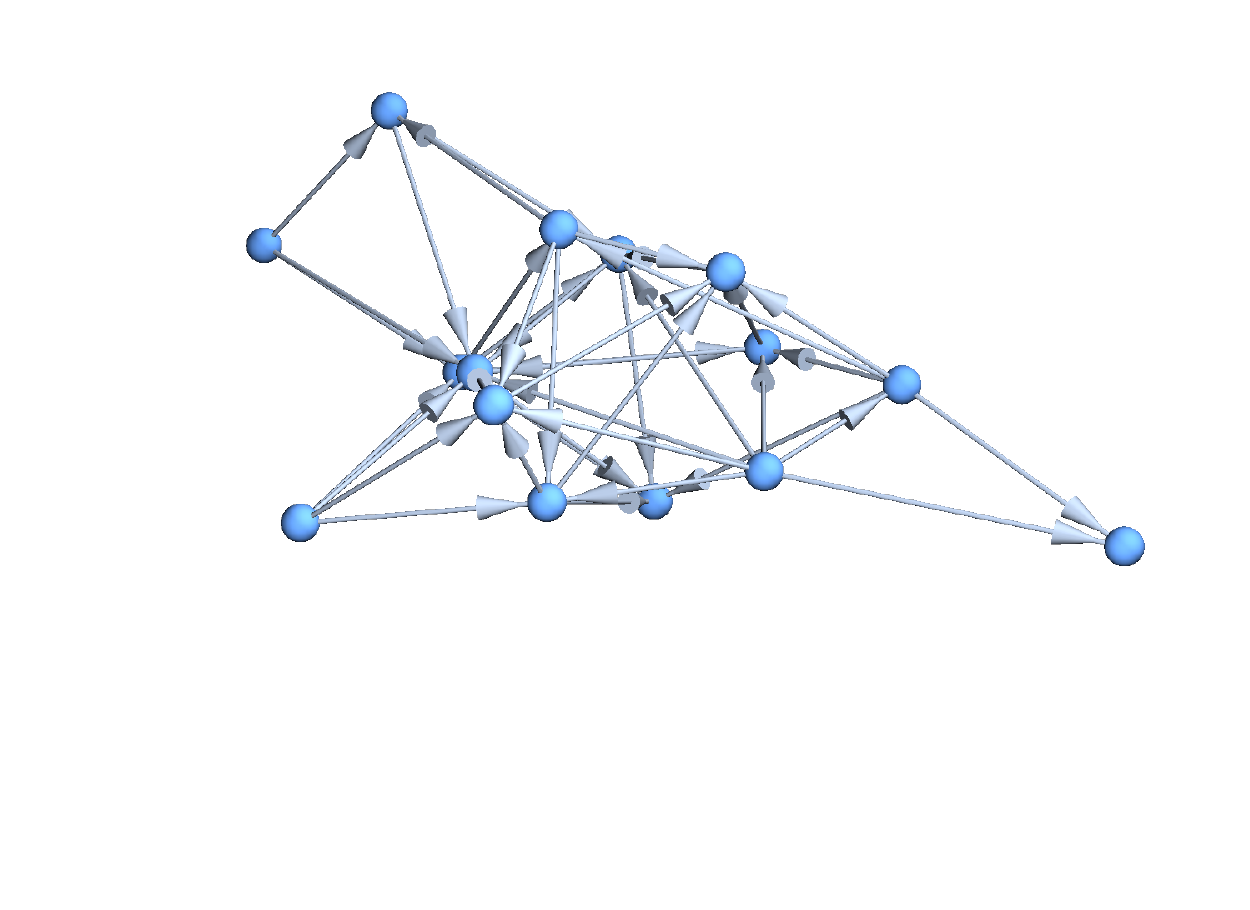}}
\label{triangle}
\caption{
Example of acyclic graphs, digraphs in which no triangle is cyclic. As we can tune
the values of a one form implementing the directions so that the curl is zero in each
triangle, this 1-form is actually given locally by a gradient field defined by a potential
function $g$.  Not globally because of cohomology but the index is then the index of 
the locally injective function $g$. 
}
\end{figure}

\paragraph{}
Example: Given a graph $G$ equipped with a direction field, we can look at the Barycentric
refinement $G_1$ of $G$. The vertices are the simplices of $G$ and two are connected
if one is contained in the other. If $e=a \to b$ is refined, we define a direction on
$e_1: a \to b \to c$, keeping the flow there in the same direction. For any 
other new connection $x \to y$, keep flowing to the lower dimensional part. This 
vector field on $G_1$ has no cycles so that the above result holds for the
{\bf Barycentric refined field}. 

\begin{figure}[!htpb]
\scalebox{0.5}{\includegraphics{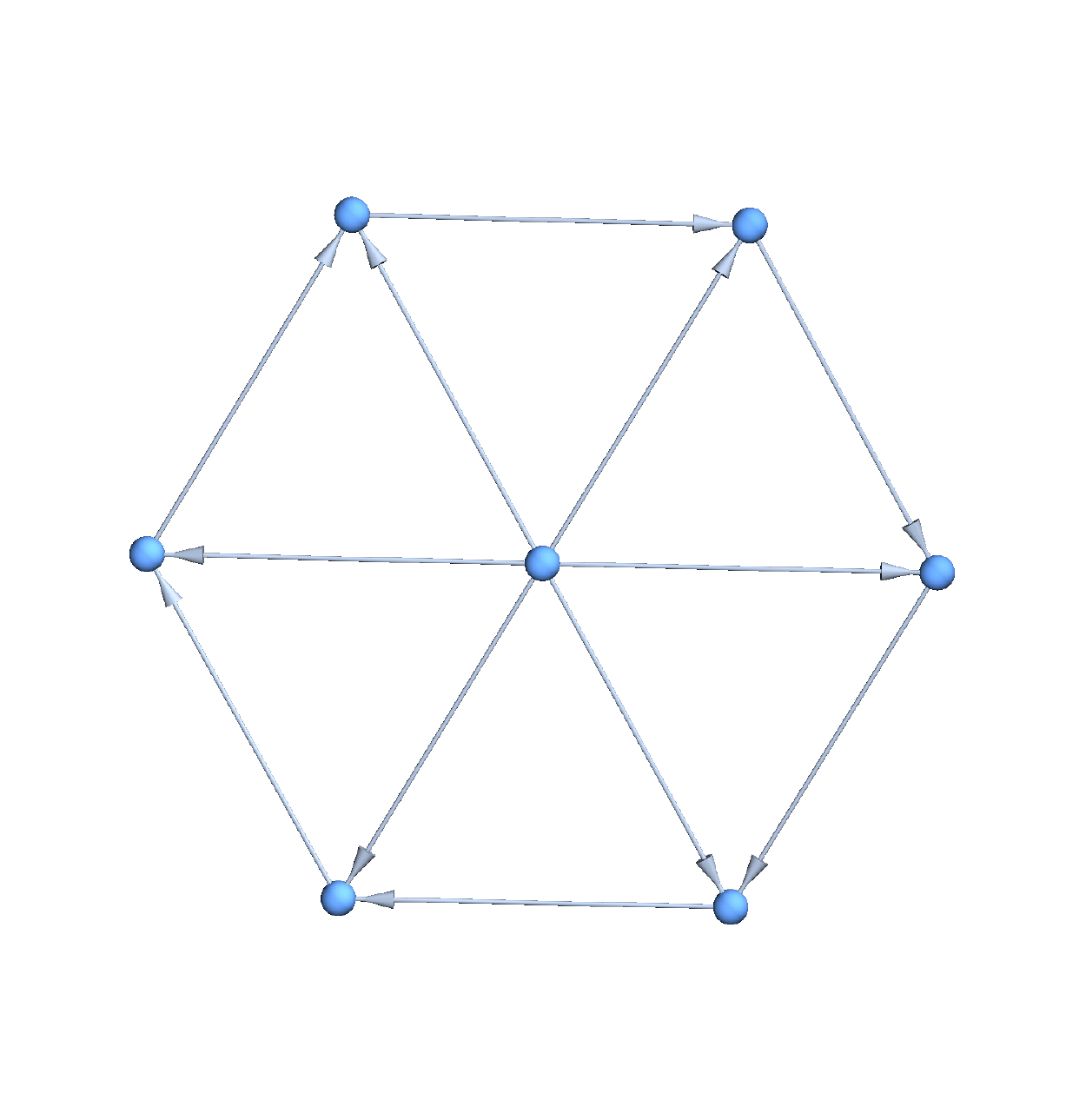}}
\label{triangle}
\caption{
The Barycentric refinement of the circular triangle splits each edge
into two, leaves the direction there and connects the newly added simplex
to all the vertices. This Barycentric refined field is always acyclic so that
the above proposition applies. 
}
\end{figure}

\begin{lemma}
The Barycentric refined field $F$ is acyclic. 
\end{lemma} 

\begin{coro}
The Poincar\'e-Hopf index function defined by a field $F$
adds up to the Euler characteristic. 
\end{coro}

\begin{figure}[!htpb]
\scalebox{0.5}{\includegraphics{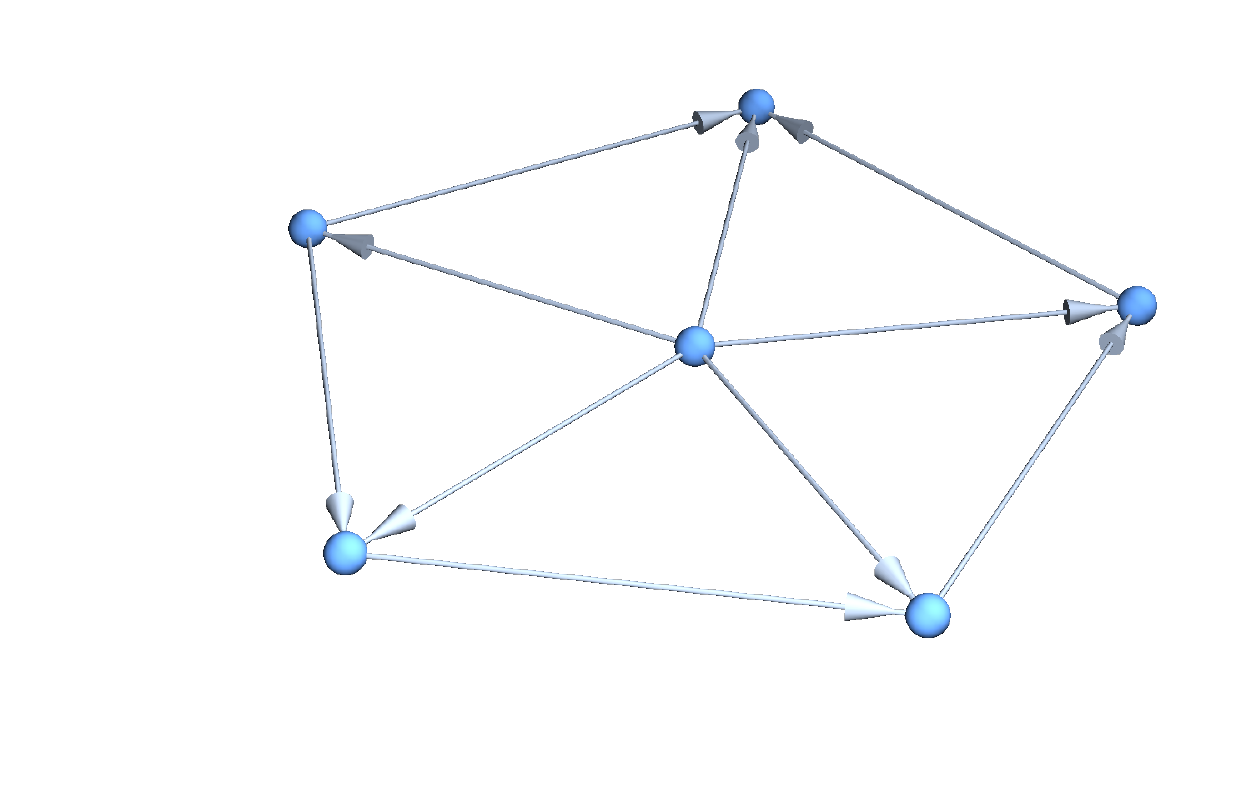}}
\scalebox{0.5}{\includegraphics{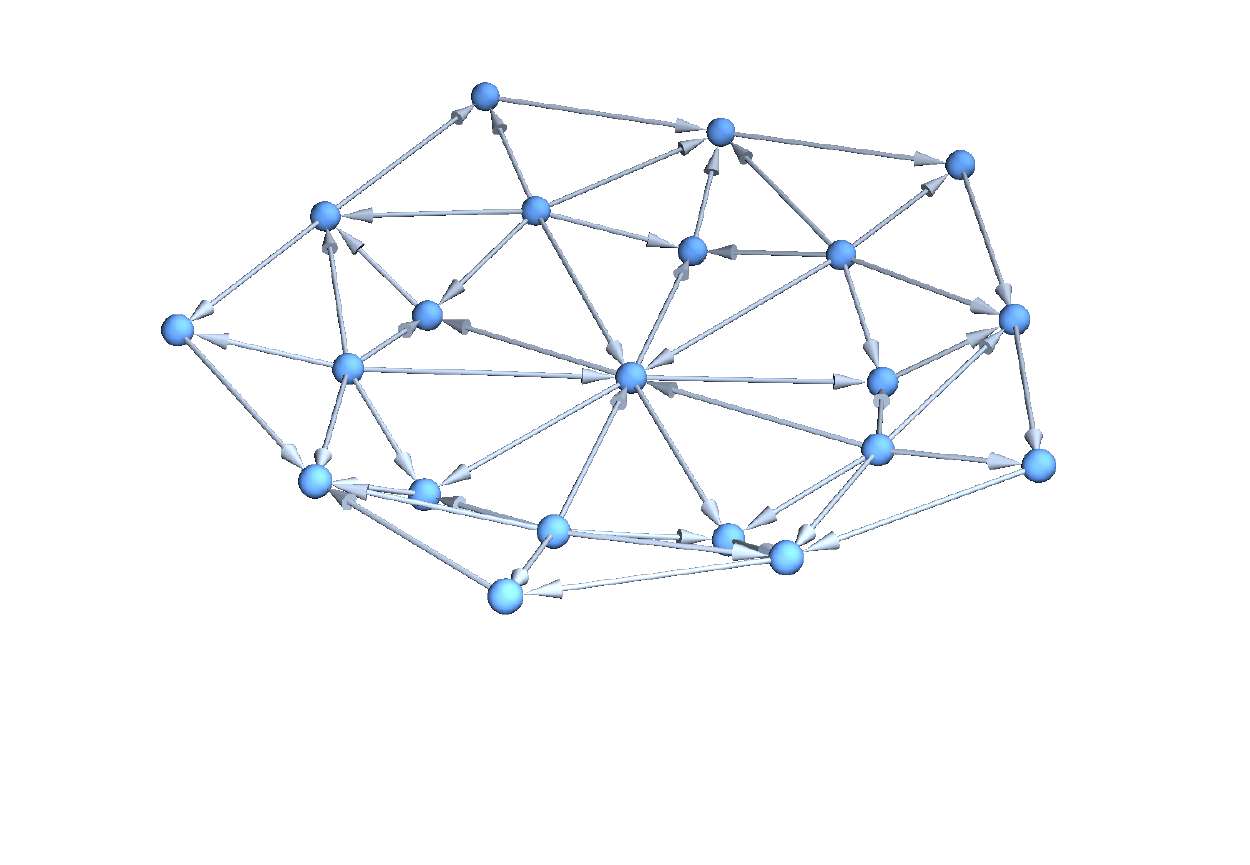}}
\label{triangle}
\caption{
A directed wheel graph and its Barycentric refinement. The later
is irrotational so that we can assign Poincar\;e-Hopf indices
to its vertices. 
}
\end{figure}

\paragraph{}
Given a simplex $x = (x_0,\cdots,x_k)$ in a graph, a {\bf local Barycentric refinement} at $x$ 
adds a new vertex $x$ and connects it to all points in $S(x) = S(x_0) \cap \cdots \cap S(x_k)$, 
where $S(x_k)$ are the unit spheres of $S_k$ as well as all points $x_k$.  
We use it here for triangle refinements when $k=2$. 

\paragraph{}
There are two ways to think about the refinement. We can stay in the category of graphs but this
changes the maximal dimension of the Whitney simplicial complex defined by $G$. If $x$ has
dimension larger than $1$, We can also 
interpret it as a refinement of simplicial complexes and keep so the dimension the same. 
For example, if $G$ is a single triangle, then refining it we can within graph theory see the result
as a 3-simplex with $f$-vector $(4,6,4,1)$. 
When getting rid of the original simplex, we have a complex with $3$ triangles and $6$ vertices
and have an $f$-vector $(3,6,4)$. 
For $k=1$, the later version gives an {\bf edge refinement} which when reversed is 
{\bf edge collapse}.  In general, we have:

\begin{lemma}
Any local Barycentric refinement $G \to G_1(x)$  at a 
simplex $x$ preserves Euler characteristic. 
\end{lemma} 

\begin{proof}
If $k$ is even, then the boundary sphere complex $S(x)$ 
has Euler characteristic $0$, the same number of even and odd dimensional 
simplices. When adding a new vertex, every odd simplex in $S(x)$ adds an
even dimensional simplex in $G_1(x)$ and every even simplex in $S(x)$ adds
an odd dimensional simplex in $G_1(x)$. We also add a $(k+1)$-dimensional simplex
$(x,x_0,... ,x_k)$ and a zero dimensional one which cancel. 
If $k$ is odd, then the boundary sphere complex $S(x)$ has
Euler characteristic $2$, There are 2 more even dimensional simplices than
odd dimensional ones. This means we add $2$ more odd dimensional simplices
to $G_1(x)$ than odd dimensional ones. But we also add an even dimensional 
$(k+1)$ simplex $(x,x_0, \cdots ,x_k)$ and vertex $x$. In total, we again do not change
the balance of even and odd dimensional simplices, what the Euler characteristic is. 
\end{proof}

\paragraph{}
Given a graph $G$ and a field $F$, we make a local Barycentric refinement to
every cyclic triangle $x$ in $G$ and make every arrow from the new vertex $x$
go away from $x$. Now, all triangles are acyclic and we can use the proposition 
to get an index. The index function is $1$ on triangles which are cyclic and $0$ 
else. The {\bf index} of a vertex $v \in V$ is defined as
$$ i_F(v) = 1-\chi(S^-_1(v)) $$ 
where $S_1^-(v) = \{ w \in S_1(x) \; F$ points from $v$ to $w \}$ and
$S_1(v)$ is the unit sphere in the Barycentric refined case. 

\paragraph{}
An other possibility to break a cyclic triangle is to make an edge refinement at one
of the edges. Similarly as the triangle refinement, also this refinements depend
in general on the order in which the refinements are done.

\section{Averaging fields}

\paragraph{}
When averaging the Poincar\'e-Hopf indices over a probability space of locally injective
functions $g$, we got curvature. The curvature depends on the probability measure.
Having choice in averaging produces flexibility which allows to 
implement natural curvatures for convex polyhedra. Instead of averaging
over all possible functions $g$, we could also average over all possible 
base maps $F: G \to V$. If for every simplex $x \in G$, and every $v \in x$
the probability that $F(x)=v$ is independent of $v$, then we get a curvature.

\paragraph{}
The old story can can be understood as starting with an 
energy $\omega(x) = (-1)^{{\rm dim}(x)}$ on the simplices of the graph, for
which $\sum_x \omega(x) = \chi(G)$ is the definition of Euler characteristic. 
The energy $\omega(x)$ is the Poincar\'e-Hopf index of the dimension function
${\rm dim}$ on the vertex set of the Barycentric refinement of $G$. The Poincar\'e-Hopf
index $i_G(x)$ of the original graph with respect to some function $g$ 
is then obtained by moving every energy from a simplex $x$ to the vertex in $x$, 
where $g$ is maximal. One gets the curvature of $G$ by distributing the energy 
$\omega(x)$ equally to every vertex of $x$. 

\paragraph{}
If the field $F$ is irrational, then it is a gradient field in every simplex. There
is therefore an ordering of the simplices. The process of distributing the value 
$\omega(x)$ of a simplex to the largest vertex in that ordering produces the index.

\paragraph{}
Here is a remark. Let $\Omega$ be the probability space of all 
fields $F$ which are irrotational. Put the uniform measure on it and call 
$E[\cdot ]$ the expectation. Denote by 
$$ K(x) = 1+\sum_{k=0} (-1)^k f_k(S(x))/(k+1) $$ 
the {\bf Levitt curvature} on the vertices of the graph, where $f_k(A)$ counts the 
number of $k$-dimensional simplices
of $A$ and $S(x)$ is the unit sphere. 

\begin{lemma}
${\rm E}[i_g] = K$. 
\end{lemma}

\begin{proof}
We only have to show that for every simplex $x$, and every two 
vertices $v,w$ in $x$ the probability that $v$ is the largest element
in $F$ is the same than that $w$ is the largest element. 
\end{proof}

\section{The hyperbolic case}

\paragraph{}
The discussion about Non-standard analysis should already indicate why every 
result which holds in the continuum also has an analogue in the discrete. 
Discrete Morse theory and especially \cite{FormanVectorfields} illustrates this. 
This is no surprise as a geometric realization of a discrete simplicial complex
is a continuum space with the same features. Still, one has to look at the discrete
case independent of the continuum and investigate how results from the continuum 
can be obtained without geometric realization. 

\paragraph{}
Let us start with a special Morse case which been mentioned a couple of times: 
if $G$ is a simplicial complex, we can define a graph $\Gamma=(V,E)$
where $V$ are the set of vertices in $G$ and $E$ the set of pairs $(x,y)$
with either $x \subset y$ or $y \subset x$. The Whitney complex $G_1$ of $\Gamma$
is called the Barycentric refinement of $G$. There is a natural choice for every
simplex $x$ in $G_1$. It is the maximal simplex $v$ in $x$. In that case,
$i_F(x) = (-1)^{{\rm dim}}(x)$ and every vertex is a critical point.
The Poincar\'e-Hopf theorem then just tells that the Euler characteristic of
the Barycentric refinement $G_1$ is the same than the Euler characteristic of $G$.

\paragraph{}
The notion of discrete manifolds can be formulated elegantly in an inductive way. 
A {\bf $d$-graph} is a finite simple graph for which all unit spheres $S(x)$ are
$(d-1)$-spheres. A {\bf $d$-sphere} $G$ is a $d$-graph such that $G-x$ is contractible for
some $x \in G$. A graph is {\bf contractible} if there exists $x$ such that $S(x)$ and $G-x$
are both contractible. The inductive definitions are primed by the assumption that
the empty graph $0$ is the $(-1)$-sphere and $1=K_1$ is contractible.
In this case, we have defined a function $g:V \to \mathbb{R}$ to be Morse, if the
central manifolds $B_g(x)$ at every point are either empty, a sphere or a product of
two spheres. In that case the symmetric index $j(x) = [i_g(x)+i_{-g}(x)]/2$ is given
in terms of the central manifold \cite{indexformula}. 

\paragraph{}
For a Morse function $g$, the gradient flow $x'=\nabla g(x)$ is a hyperbolic
system. The generalization in the continuum
to $d$-dimensional smooth manifold is given by
the Sternberg-Grobman-Hartman linearization theorem assures that new a 
hyperbolic equilibrium point $x$ of a differential equation $x'=F(x)$, the 
stable and unstable manifolds $W^{\pm}(x)$ intersect a small sphere $S_r(x)$ 
in $m-1$ and $(d-m-1)$-dimensional spheres. The Poincar\'e-Hopf index $i_F(x)$ is 
$(-1)^m$, where $m=m(x)$ is the Morse index, the dimension of the stable manifold at $x$. 
Poincar\'e-Hopf is then usually formulated as $\chi(M) = \sum_{k} (-1)^k c_k$, where $c_k$
is the number of equilibria with Morse index $k$. By using the dynamics to build a cell
complex, one has also a bound on the Betti numbers $b_m \leq c_m$. 

\paragraph{}
In discrete set-up, we have to assume a bi-directed complex which comes from a $d$-graph.
For every $v \in V$ we have subsets $F^{\pm}(v) \subset G$. A vector field is then
called {\bf hyperbolic} if for every $v$ for which $F^{\pm}(v)$ is not contractible, the 
sets $F^-(v)$ is homotopic to a $(m-1)$-sphere and $F^+(v)$ to a $d-m-1$ sphere and that the join of these two
spheres is homotopic to a $(d-1)$-sphere, the unit sphere in $M$. 
For example, in the case $d=2$, we have either sinks with $m=2$, sources
with $d=0$ or hyperbolic saddle points with $m=1$. In the later case, the sets $F^{\pm}(x)$ are
homotopic to the zero sphere and the join of these two zero spheres is a 1-sphere. 
Note however, that the Betti inequalities $b_m \leq c_m$ are no more true in general (this is the
same as in the continuum). For a circle for example, there are vector fields without any critical points. 
The Betti of the circle are however $b_0=1, b_1=1$. It is only in the case of gradient fields
that we always have a critical point of index $m=0$ (a minimum) and a critical point with index $m=1$
(a maximum). The Reeb sphere theorem (see \cite{knillreeb} for a discussion in the discrete) assures
that spheres are characterized by the existence of Morse functions with exactly two critical points. 

\bibliographystyle{plain}

\begin{thebibliography}{10}

\bibitem{Dlotko2014}
P.~Dlotko and H.~Wagner.
\newblock Simplification of complexes of persistent homology computations.
\newblock {\em Homology, Homotopy and Applications}, 16:49--63, 2014.

\bibitem{Evako2013}
A.V. Evako.
\newblock The {Jordan-Brouwer} theorem for the digital normal n-space space
  {$Z^n$}.
\newblock http://arxiv.org/abs/1302.5342, 2013.

\bibitem{FormanVectorfields}
R.~Forman.
\newblock Combinatorial vector fields and dynamical systems.
\newblock {\em Math. Z.}, 228(4):629--681, 1998.

\bibitem{forman98}
R.~Forman.
\newblock Morse theory for cell complexes.
\newblock {\em Adv. Math.}, page~90, 1998.

\bibitem{Forman1999}
R.~Forman.
\newblock Combinatorial differential topology and geometry.
\newblock {\em New Perspectives in Geometric Combinatorics}, 38, 1999.

\bibitem{HermanDigitalSpaces}
G.T. Herman.
\newblock {\em Geometry of digital spaces}.
\newblock {Birkh\"auser}, {B}oston, {B}asel, {B}erlin, 1998.

\bibitem{HopfVectorfields}
H.~Hopf.
\newblock Vektorfelder in {$n$}-dimensionalen {M}annigfaltigkeiten.
\newblock {\em Math. Ann.}, 96(1):225--249, 1927.

\bibitem{Lanford1986}
O.E.~Lanford III.
\newblock An introduction to computers and numerical analysis.
\newblock In {\em {Ph\'enom\`enes critiques, syst\`emes al\'eatoires, th\'ories
  de jauge, Part I, II, Les Houches, 1984}}, pages 1--86. North-Holland,
  Amsterdam, 1986.

\bibitem{I94a}
A.V. Ivashchenko.
\newblock Graphs of spheres and tori.
\newblock {\em Discrete Math.}, 128(1-3):247--255, 1994.

\bibitem{cherngaussbonnet}
O.~Knill.
\newblock A graph theoretical {Gauss-Bonnet-Chern} theorem.
\newblock {\\}http://arxiv.org/abs/1111.5395, 2011.

\bibitem{poincarehopf}
O.~Knill.
\newblock A graph theoretical {Poincar\'e-Hopf} theorem.
\newblock {\\} http://arxiv.org/abs/1201.1162, 2012.

\bibitem{indexformula}
O.~Knill.
\newblock An index formula for simple graphs \hfill.
\newblock {\\}http://arxiv.org/abs/1205.0306, 2012.

\bibitem{indexexpectation}
O.~Knill.
\newblock On index expectation and curvature for networks.
\newblock {\\}http://arxiv.org/abs/1202.4514, 2012.

\bibitem{knillcalculus}
O.~Knill.
\newblock {The theorems of Green-Stokes,Gauss-Bonnet and Poincare-Hopf in Graph
  Theory}.
\newblock {\\}http://arxiv.org/abs/1201.6049, 2012.

\bibitem{eveneuler}
O.~Knill.
\newblock The {E}uler characteristic of an even-dimensional graph.
\newblock {{\\}http://arxiv.org/abs/1307.3809}, 2013.

\bibitem{KnillBaltimore}
O.~Knill.
\newblock Classical mathematical structures within topological graph theory.
\newblock {{\\}http://arxiv.org/abs/1402.2029}, 2014.

\bibitem{Unimodularity}
O.~Knill.
\newblock On {F}redholm determinants in topology.
\newblock {\\}https://arxiv.org/abs/1612.08229, 2016.

\bibitem{AmazingWorld}
O.~Knill.
\newblock The amazing world of simplicial complexes.
\newblock {{\\}https://arxiv.org/abs/1804.08211}, 2018.

\bibitem{CartanMagic}
O.~Knill.
\newblock Cartan's magic formula for simplicial complexes.
\newblock {{\\}https://arxiv.org/abs/1811.10125}, 2018.

\bibitem{dehnsommervillegaussbonnet}
O.~Knill.
\newblock Dehn-{S}ommerville from {G}auss-{B}onnet.
\newblock {\\}https://arxiv.org/abs/1905.04831, 2019.

\bibitem{EnergizedSimplicialComplexes}
O.~Knill.
\newblock Energized simplicial complexes.
\newblock {\\}https://arxiv.org/abs/1908.06563, 2019.

\bibitem{EnergyComplex}
O.~Knill.
\newblock The energy of a simplicial complex.
\newblock {\\}https://arxiv.org/abs/1907.03369, 2019.

\bibitem{parametrizedpoincarehopf}
O.~Knill.
\newblock A parametrized {P}oincare-{H}opf theorem and clique cardinalities of
  graphs.
\newblock {\\}https://arxiv.org/abs/1906.06611, 2019.

\bibitem{knillreeb}
O.~Knill.
\newblock A {R}eeb sphere theorem in graph theory.
\newblock {\\}https://arxiv.org/abs/1903.10105, 2019.

\bibitem{lanford98}
Oscar~E. Lanford, III.
\newblock Informal remarks on the orbit structure of discrete approximations to
  chaotic maps.
\newblock {\em Experiment. Math.}, 7(4):317--324, 1998.

\bibitem{LawlerInternalSetTheory}
G.~F. Lawler.
\newblock Comments on edward nelson's internal set theory: a new approach to
  nonstandard analysis.
\newblock {\em Bulletin (New Series) of the AMS}, 4:503--506, 2011.

\bibitem{Nelson77}
E.~Nelson.
\newblock Internal set theory: A new approach to nonstandard analysis.
\newblock {\em Bull. Amer. Math. Soc}, 83:1165--1198, 1977.

\bibitem{Nelson}
E.~Nelson.
\newblock {\em Radically elementary probability theory}.
\newblock Princeton university text, 1987.

\bibitem{NelsonSimplicity}
E.~Nelson.
\newblock The virtue of simplicity.
\newblock In {\em The Strength of Nonstandard Analysis}, pages 27--32.
  Springer, 2007.

\bibitem{Poincare1885}
H.~Poincar{\'e}.
\newblock Sur les courbes {d\'efinies} par les \'equations differentielles.
\newblock {\em Journ. de Math}, 4, 1885.

\bibitem{Ran74}
F.~Rannou.
\newblock Numerical study of discrete area-preserving mappings.
\newblock {\em Acta Arithm}, 31:289--301, 1974.

\bibitem{Robert}
A.~Robert.
\newblock {\em Analyse non standard}.
\newblock Presses polytechniques romandes, 1985.

\bibitem{RomeroSergeraert}
A.~Romero and F.~Sergeraert.
\newblock Discrete vector fields and fundamental algebraic topology.
\newblock Version 6.2. University of Grenoble, 2012.

\bibitem{Spivak1999}
M.~Spivak.
\newblock {\em A comprehensive Introduction to Differential Geometry I-V}.
\newblock Publish or Perish, Inc, Berkeley, third edition, 1999.

\bibitem{ZhVi88}
X-S. Zhang and F.~Vivaldi.
\newblock Small perturbations of a discrete twist map.
\newblock {\em Ann. Inst. H. Poincar\'e Phys. Th\'eor.}, 68:507--523, 1998.

\end{thebibliography}

\end{document}